\documentclass{ws-ijfcs}
\usepackage{enumerate}
\usepackage{url}
\begin{document}

\markboth{Yen Hung Chen}
{The clustered selected-internal Steiner tree problem}
\input{epsf}
%
\catchline{}{}{}{}{}
%

\title{The clustered selected-internal Steiner tree problem \footnote{This work was supported in part by the Ministry of Science and Technology Contract grant
number: MOST 107-2115-M-845-001 and 109-2221-E-845-001.}}

\author{Yen Hung Chen\footnote{corresponding author.}}

\address{Department of Computer Science, University of Taipei, No.1, Ai-Guo West Road \\
Taipei 10048, Taiwan. \\
\email{yhchen@utaipei.edu.tw}}

\maketitle

\begin{history}
\received{(Day Month Year)}
\accepted{(Day Month Year)}
\comby{(xxxxxxxxxx)}
\end{history}

\begin{abstract}
Given a complete graph $G=(V,E)$, with nonnegative edge costs, two subsets $R \subset V$ and $R^{\prime} \subset R$, a partition $\mathcal{R}=\{R_1,R_2,\ldots,R_k\}$ of $R$, $R_i \cap R_j=\phi$, $i \neq j$  and $\mathcal{R}^{\prime}=\{R^{\prime}_1,R^{\prime}_2,\ldots,R^{\prime}_k\}$ of $R^{\prime}$, $R^{\prime}_i \subset R_i$, 
 a clustered Steiner tree is a tree $T$ of $G$ that spans all vertices in $R$ 
such that $T$ can be cut into $k$ subtrees $T_i$ by removing $k-1$ edges and each subtree $T_i$ spanning all vertices in $R_i$, $1 \leq i \leq k$. 
The  cost of a clustered Steiner tree is defined to be the sum of the costs of all its edges.
  A clustered selected-internal Steiner tree of $G$ is a clustered Steiner tree for $R$ if all vertices in $R^{\prime}_i$ are 
internal vertices of $T_i$, $1 \leq i \leq k$. The clustered selected-internal Steiner tree problem   is concerned with  the determination of a clustered selected-internal  Steiner tree $T$   for $R$ and  $R^{\prime}$ in $G$ with minimum cost.  
 In this paper, we present the first known approximation algorithm with performance ratio $(\rho+4)$  
for the clustered selected-internal Steiner tree problem,   
where $\rho$ is the best-known performance ratio for the Steiner tree problem.  
\end{abstract}


\section{Introduction}

Given an undirected graph $G=(V,E)$, a subset $R \subseteq V$ of vertices, and  a nonnegative edge cost function, a {\em Steiner tree} is used to find  a tree in $G$ that spans all vertices in $R$. 
The given vertices $R$ are usually referred to as {\em terminals} and  other vertices $V \setminus R$ as {\em Steiner} vertices. 
The cost of a Steiner tree is defined to be the sum of the costs of all its edges. 
The Steiner tree problem is  concerned with  the determination of a Steiner tree   for $R$  in $G$ with minimum cost~\cite{cheng,du0,Du,hwang}. 
 The Steiner tree problem had been showed to be NP-Complete~\cite{npcstp} and MAX SNP-hard~\cite{Bern}. Hence, many 
approximation algorithms  had been designed for the Steiner tree problem~\cite{berman,borcher,Byrka,Hougardy1,Hougardy2,prom,Robins05,zel93,zel932}.  Moreover, 
The Steiner tree problem had many significant  applications in  network routing, VLSI design, and phylogenetic tree reconstruction~\cite{vlsi1,cheng,drummond,du0,Du,bio1,hwang,sai,winter}.

Motivated  by the applications of the facility allocation in the (sensor) network and  engineering change orders (ECO) in VLSI design,  
Hsieh and Yang~\cite{Hsieh2} proposed a variant  of the Steiner tree problem, called the {\em selected-internal  Steiner tree problem}.  
Given a complete undirected graph $G=(V,E)$, a nonnegetive cost function on edges, and two  subsets $R \subseteq V$ and  $R^{\prime} \subset  R$,  
a Steiner tree for $R$ in $G$ is a selected-internal  Steiner tree if all  terminals in $R^{\prime}$ are internal vertices of this Steiner tree.   
The {selected-internal  Steiner tree problem} (SISTP  for short) is concerned with  the determination of a selected-internal  Steiner tree   for $R$ and  $R^{\prime}$ in $G$ with minimum cost~\cite{Hsieh2,Li}.  
For the SISTP, 
without loss of generality, we assume  $|R \setminus R^{\prime}|\geq 2$ for the  SISTP, otherwise the solution of SISTP may not exist. 
Then Hsieh and Yang~\cite{Hsieh2}  
showed that the  SISTP is NP-complete and 
MAX SNP-hard. They also proposed a $2\rho$-approximation algorithm for the SISTP
 on {metric graphs}  
(i.e., a complete graph and the lengths of edges satisfy the triangle inequality),   
where $\rho$ is the best-known performance ratio for the Steiner tree problem  whose performance ratio is $\ln {4}+\epsilon\approx 1.39$~\cite{Byrka}.   
 Li et al.~\cite{Li}  
 improved the performance ratio to   $(\rho + 1)$ for the  SISTP.

Although the SISTP is defined by a group  of vertices (terminals),  
some  applications of computer and transportation network routing are considered in  more than one group  of vertices~\cite{bao,chisman,ding,gutt,lap,wu}. 
Chisman~\cite{chisman} 
 presented  a variant  of the traveling salesman problem~\cite{cormen,gol}, 
 called as the clustered traveling salesman problem. Wu and Lin~\cite{wu} 
 proposed another variant  of the Steiner tree problem, called as the clustered Steiner tree  problem.  Given a complete graph $G=(V,E)$, with  a nonnegetive cost function on edges, a subset $R \subset V$, and a partition $\mathcal{R}=\{R_1,R_2,\ldots,R_k\}$ of $R$, $R_i \cap R_j=\phi$, $i \neq j$, a clustered Steiner tree is a tree $T$ of $G$ that spans all vertices in $R$ 
such that $T$ can be cut into $k$ subtrees $T_i$ by removing $k-1$ edges and each subtree $T_i$ spanning all vertices in $R_i$, $1 \leq i \leq k$. In other word, all the vertices in the same cluster ($R_i$) are clustered together in $T$. 
Each subtree $T_i$ is called as a local tree of $T$. The  cost of a clustered Steiner tree is defined to be the sum of the costs of all its edges.  The local cost of a clustered Steiner tree is  the sum of the costs of all its edges in all its local trees.  Then the inter-cluster cost of a clustered Steiner tree is  the sum of
  the costs of remaining edges.  The clustered  Steiner tree problem (CSTP for short) 
is concerned with   the determination of  a clustered Steiner tree $T$ for $R$ in $G$  with minimum cost. 
 Wu and Lin~\cite{wu} 
 showed that the  CSTP is NP-complete and proposed a $(\rho+2)$-approximation algorithm for the CSTP on {metric graphs}. 
In this paper, we presented a variant of the SISTP and the CSTP,  called as the  clustered selected-internal  Steiner tree  problem.  
Given a complete graph $G=(V,E)$, with   a nonnegetive cost function on edges, two subsets $R \subset V$ and $R^{\prime} \subset R$, a partition $\mathcal{R}=\{R_1,R_2,\ldots,R_k\}$ of $R$, $R_i \cap R_j=\phi$, $i \neq j$,  
and $\mathcal{R}^{\prime}=\{R^{\prime}_1,R^{\prime}_2,\ldots,R^{\prime}_k\}$ of $R^{\prime}$, $R^{\prime}_i \subset R_i$,   a clustered selected-internal Steiner tree of $G$ is a clustered Steiner tree for $R$ if all vertices in $R^{\prime}_i$ are  
internal vertices of $T_i$, $1 \leq i \leq k$. 
The clustered selected-internal Steiner tree problem (CSISTP for short)  is concerned with  the determination of a clustered selected-internal  Steiner tree $T$   for $R$ and  $R^{\prime}$ in $G$ with minimum cost. It is not hard to see the CSISTP is NP-hard, since the SISTP is its special versions when $k=1$. 
  A Possible application of the CSISTP is to combine the applications of the SISTP and the CSTP in the following scenarios. 
Suppose there is a group of $|R|$ hosts (servers) in a computer network.  
A multicast tree is  about building a tree to  connect the group such that data  can be transmitted to the group.  
In some network  resource allocation strategies, some specified  
hosts (servers) in the group must act  as transmitters and the others 
need not have this restriction~\cite{Hsieh2}.  
 Hence,  transmitters are represented by the internal  vertices  of 
the multicast tree. The cost of an edge of the multicast tree represents 
the transmission  distance, building or routing costs between two hosts 
in the network.  Hence, 
a multicast tree in a  network  whose some specified servers in the group 
must be transmitters can be modeled by  
the SISTP~\cite{Hsieh2}. Then, for some communication networks, sometimes 
the edges are divided into two levels: inter-cluster or
intra-cluster, possibly with different costs, qualities, and capacities. 
After the multicast tree is constructed, the communications between hosts in the same cluster should be routed
locally rather than globally for the sake of capacity consideration or the simpleness
of routing protocols~\cite{wu}. If all  local topologies are given~\cite{wu}. The purpose is to design the inter-cluster topology, as well as the possible insertion of local Steiner vertices without violating their topologies. These reasons caused us to build a multicast tree which satisfies the definition of the CSTP and the  SISTP, simultaneously.  

In this paper,  we  design the first known approximation algorithm with  performance ratio of $(\rho+4)$ for the CSISTP on metric graphs. 
The rest of this paper is organized as follows. 
In Section 2, we describe our $(\rho+4)$-approximation algorithm to solve the CSISTP.  Finally, we  give the concluding remarks in Section 3. 

\section{A $(\rho+4)$-Approximation Algorithm for the CSISTP}

Formally, we list the definition of the CSISTP as follows.
\begin{definition} For   a complete graph $G=(V,E)$,  a subset $R \subset V$, a partition $\mathcal{R}=\{R_1,R_2,\ldots,R_k\}$ of $R$, $R_i \cap R_j=\phi$, $i \neq j$, a {clustered Steiner tree} is a tree $T$ of $G$ that spans all vertices in $R$ 
such that $T$ can be cut into $k$ subtrees $T_i$ by removing $k-1$ edges and each subtree $T_i$ spanning all vertices in $R_i$, $1 \leq i \leq k$. 
\end{definition}

\begin{description}
\item[\textbf{CSTP}] (Clustered Steiner Tree Problem)
\item[Instance:]  A complete graph $G=(V,E)$ with  a  cost function  
$c: E \rightarrow \mathbb{R}_{\geq 0}$ on the edges, a subset $R \subset V$, a partition $\mathcal{R}=\{R_1,R_2,\ldots,R_k\}$ of $R$, $R_i \cap R_j=\phi$, $i \neq j$. 
\item[Problem:]  Find a clustered Steiner tree $T$ for $R$ in $G$ such that the sum of the costs of all its edges in $T$  is minimized. 
\end{description}

\begin{definition}
 For   a complete graph $G=(V,E)$, two subsets $R \subset V$ and $R^{\prime} \subset R$, a partition $\mathcal{R}=\{R_1,R_2,\ldots,R_k\}$ of $R$, $R_i \cap R_j=\phi$, $i \neq j$,  and $\mathcal{R}^{\prime}=\{R^{\prime}_1,R^{\prime}_2,\ldots,R^{\prime}_k\}$ of $R^{\prime}$, each $R^{\prime}_i \subset R_i$.
 A {clustered selected-internal Steiner tree} of $G$ is a clustered Steiner tree for $R$ if all vertices in $R^{\prime}_i$ are internal vertices of $T_i$, $1 \leq i \leq k$. 
\end{definition}

\begin{description}
\item[\textbf{CSISTP}] (Clustered Selected-Internal  Steiner Tree Problem)
\item[Instance:]  A complete graph $G=(V,E)$ with  a  cost function  
$c: E \rightarrow \mathbb{R}_{\geq 0}$ on the edges, two subsets $R \subset V$ and $R^{\prime} \subset R$, a partition $\mathcal{R}=\{R_1,R_2,\ldots,R_k\}$ of $R$, $R_i \cap R_j=\phi$, $i \neq j$,  
and $\mathcal{R}^{\prime}=\{R^{\prime}_1,R^{\prime}_2,\ldots,R^{\prime}_k\}$ of $R^{\prime}$, each $R^{\prime}_i \subset R_i$.
\item[Problem:]  Find a clustered selected-internal  Steiner tree  $T$ for $R$ and  
$R^{\prime}$ in $G$ such that the sum of the costs of all its edges in $T$  is minimized.. 
\end{description}

For a subgraph $T$ of $G$, the cost of a tree $T=(V_T,E_T)$,  denoted by \textit{$c(T)$}, is the sum of the costs of all its edges in $T$, that is,
$c(T)=\sum_{e \in E_T} c(e)$. 
The following examples illustrate the CSTP, SISTP and  CSISTP, respectively.  Consider the instance shown in Fig.~\ref{Fig1}, in which the graph $G=(V,E)$, $R=\{A,B,C,D,E,F\}$, 
and $R^{\prime}=\{C,F\}$. An optimal solution $T_b$ of $G$ for the CSTP with $k=3$ when $R_1=\{A,B\}, R_2=\{C,D,E\}, R_3=\{F\}$ is shown in Fig.~\ref{Fig2}, in which   $c(T_b)=23$. An optimal solution $T_c$ of $G$ for the SISTP is shown in Fig.~\ref{Fig3}, in which  $c(T_c)=22$. 
An optimal solution $T_d$ of $G$ for the CSISTP with $k=2$ when $R_1=\{A,B,F\}, R_2=\{C,D,E\}$, $R_1^{\prime}=\{F\}, R_2^{\prime}=\{C\}$
 is shown in Fig.~\ref{Fig4}, in which   $c(T_d)=25$.

\begin{figure}
\centerline{\epsfbox{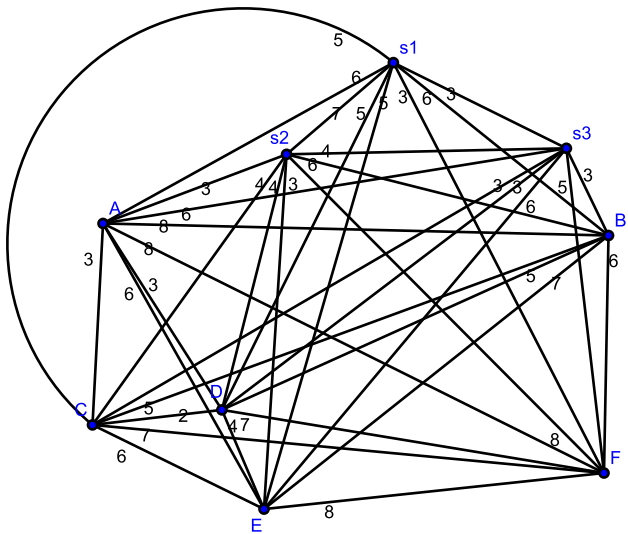}}
\caption{An instance: A complete graph $G=(V,E)$, $R=\{A,B,C,D,E,F\}$, $R^{\prime}=\{C,F\}$.}
\label{Fig1}
\end{figure}
\begin{figure}
\centerline{\epsfbox{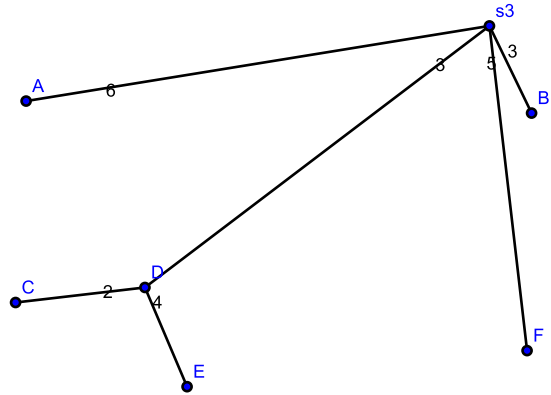}}
\caption{The optimal solution $T_b$ when $R_1=\{A,B\}, R_2=\{C,D,E\}, R_3=\{F\}$ for the CSTP.  (Note that $c(T_b)=23$).}
\label{Fig2}
\end{figure}
\begin{figure}
\centerline{\epsfbox{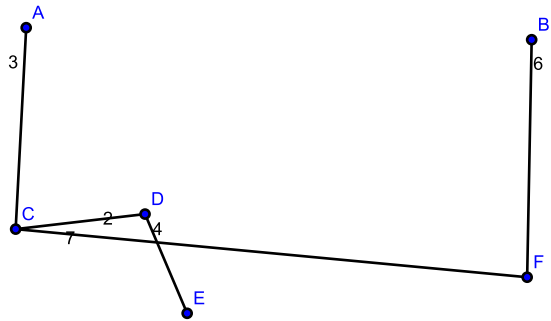}}
\caption{The optimal solution $T_c$ for the SISTP. (Note that $c(T_c)=22$).}
\label{Fig3}
\end{figure}

\begin{figure}
\centerline{\epsfbox{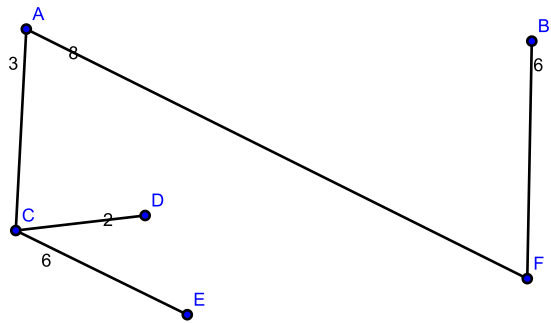}}
\caption{The optimal solution $T_d$ when $R_1=\{A,B,F\}, R_2=\{C,D,E\}$, $R_1^{\prime}=\{F\}, R_2^{\prime}=\{C\}$
for the CSISTP. (Note that $c(T_d)=25$).}
\label{Fig4}
\end{figure}

\begin{definition}
For a cluster Steiner tree $T$ spanning $R$, the local tree of $T_i$ on $T$ is the
minimal subtree of $T$ that spans all vertices in $R_i$, $1\le i\le k$. 
\end{definition}

In this section, we  present a $(\rho+4)$-approximation algorithm   for  
the CSISTP, whose cost function is metric. For a complete graph $G=(V,E)$, let $c(u,v)$ 
denote the cost of a edge $(u,v)$, for the two vertices $u,v \in V$.  
For a vertex subset $S$, let $G[S]$  denote the subgraph of $G$ induced by $S$. 
A  minimum spanning tree (MST for short)~\cite{cormen,prime} of $G[S]$ is  concerned with  the determination of a tree spanning  
  all verteices in $S$ with minimum cost of $G[S]$. We let \textit{$A_{mst}(G[S])$} be the algorithm to find a minimum spanning tree of $G[S]$. 
For a graph $H=(V_H,E_H)$, the {\em contraction} of an edge $(u, v)$ is to replace the two vertices $u$ and $v$ with a new vertex $w$, and then 
the edge cost  is assigned to $c(w,x) = \min{\{c(u, x), c(v, x)\}}$ for any other vertex $x$. For a subgraph $S=(V_S,E_S)$ in $H$, the contraction of
 $S$ in $H$ means to contract all the edges of $E_S$ in an arbitrary order, and the resulting graph is denoted by $H/S$. For convenience, 
we also use $H/S$  to $H/H[S]$ when $S$ is a vertex subset. Finally, we let $G/R$ denote the graph resulted from contracting every $R_i \in R$, $1 \le i \le k$, i.e., each cluster $R_i$ is concentrated into a vertex~\cite{wu}.

\begin{definition} For a  graph $G=(V,E)$, a  {Hamiltonian path} of $V$ is a path that visits each vertex in $V$ exactly once. 
\end{definition}
\begin{definition} A graph $G=(V,E)$ is Hamiltonian-connected if for every pair of vertices $u$ and $v$, the two vertices can be 
 connected by a Hamiltonian path from $u$ to $v$.
\end{definition}

For any clustered Steiner tree $T$ and each its local tree $T_i$  of $R_i$, Wu and Lin~\cite{wu} 
showed $T$ can be transformed  into a clustered Steiner tree $\widehat{T}$ such that each local tree $\widehat{T}_i$ of $\widehat{T}$ is a Hamiltonian Path of $R_i$. By
Wu and Lin's algorithm~\cite{wu},  for any clustered Steiner tree $T$, if the vertex $r$ in $R_i$ is an internal vertex in $T_i$,  the vertex $r$ is also an internal vertex in $\widehat{T}_i$. Hence, Wu and Lin's  algorithm~\cite{wu} also satisfies any clustered selected-internal  Steiner tree for $R$ and  $R^{\prime}$ in $G$, i.e., each vertex in $R^{\prime}_i$ is also an internal vertex in  
$\widehat{T}_i$. 

\begin{definition}~\cite{wu}
For a clustered Steiner tree $T$ for $R$, the inter-cluster tree of $T$ is that the contraction of  all its local trees becomes in a tree, denoted by $T/R$.
\end{definition}

For any clustered Steiner tree $T$ for $R$ and its local tree $T_i$, $1\le i \le k$, the next three lemmas come from~\cite{wu}. 

\begin{lemma}\label{ltree}
Each local tree $T_i$ is replaced with a Hamiltonian path $\widehat{T}_i$ of $R_i$ and 
each vertex in $R^{\prime}_i$ is  a internal vertex of $\widehat{T}_i$,  $1 \le i \le k$.
 We have $c(\widehat{T}_i) \le 2 c(T_i)$.
\end{lemma}

\begin{lemma}\label{intercluster}
For each inter-cluster tree $\widehat{T}/R$ of $\widehat{T}$, we have  $c(\widehat{T}/R) \le c(T)$.
\end{lemma}

Since each clustered Steiner tree $T$  satisfies the Lemma~\ref{ltree} and Lemma~\ref{intercluster}, we have  next Lemma. 
\begin{lemma}\label{cstp}
Let $T_{opt}$ be the optimal solution for the CSTP. There exists a clustered Steiner tree $\widehat{T}$ such that each local tree $\widehat{T}_i$ of $\widehat{T}$ 
 is a Hamiltonian Path of $R_i$ and each vertex in $R^{\prime}_i$ is  a internal vertex of $\widehat{T}_i$ that  
$c(\widehat{T}/R) \le c(T_{opt})$.
\end{lemma}

Given a connected graph $G=(V,E)$, the \textit{cube} of $G$, denoted by $G^3=(V,E_{G^3})$, is the graph with the same vertex set as $G$ and any  
edge $(u,v)$ in $E_{G^3}$  if and only if 
there exists a path  between the two vertices $u$ and $v$ in $G$ and the number of edges in the path is at most three. 
 Independently, Sekanina~\cite{seka}  
 and Karaganis~\cite{kara}  
 proved that the cube of every connected graph with at least three vertices is 
{Hamiltonian-connected}. 
 We let \textit{$A_{h}(G,u,v)$} be the 
algorithm  for finding a Hamiltonian path between the two vertices $u$ and $v$ in the $G^3$ by Karaganis' proof~\cite{kara}, whose time-complexity is 
$O(|V|^2)$. See Appendixm or Chen~\cite{chen} 
for more details of this algorithm.   
For a tree $T=(V_T, E_T)$ and $(u,w),(w,v) \in E_T$, we define the {\it shortcut} between $u$ and $v$ is to replace edges $(u,w)$ and $(w,v)$ with $(u,v)$.

\begin{lemma} \label{hamratio}
For every tree $T=(V_T,E_T)$ of $G$, if $T^h=(V_{T}, E_{T^h})$ is the output of Hamiltonian-path of $V_T$  by \textit{$A_{h}(T,u,v)$}, 
 we have $c(T^h) \le 2 c(T)$ by the  triangle inequality with doubling the tree edges $E_T$ and then traversal shortcuts between the adjacent vertices in $T^h$.
\end{lemma}

Now, we describe the  $(\rho+4)$-approximation algorithm for the CSISTP.
First, for each $R_i$,  we use Algorithm \textit{$A_{mst}(G[R_i])$} to find a MST $T^s_i$ of  $G[R_i]$, $1\le i \le k$. Next, select any two 
vertices $u$ and $v$ in  $R_i \setminus R^{\prime}_i$ (Note that  $|R_i \setminus R^{\prime}_i| \geq 2$).  Then  
 we use \textit{$A_{h}(T^{s}_i,u,v)$} to find a Hamiltonian path $T^{h}_i$ between the two vertices $u$ and $v$  in the cube of $T^{s}_i$.  
By Lemma~\ref{hamratio}, we have $c(T^{h}_i) \le 2 c(T^{s}_i)$, $1\le i \le k$. 
Moreover. we   construct $G/R$ and let  $\widetilde{R} = \{ r_i | 1 \le i \le  k \} $, where $r_i$ is the vertex resulted from the contraction of $R_i$. 
For a graph $G=(V,E)$ with a subset $R \subseteq V$, let $\mathcal{A}_{stp}(G,R)$ be  the $\rho$-approximation algorithm 
to solve the Steiner tree problem for $R$ in $G$.  
Furthermore, 
we use $\mathcal{A}_{stp}(G/R,\widetilde{R})$ to find an inter-cluster tree that spans all vertices in $\widetilde{R}$.
Finally, replace each $r_i$ with  $T^{h}_i$ to obtain a clustered Steiner tree (also a clustered selected-internal Steiner tree).

For clarification, we describe the approximation algorithm for the CSISTP as follows.

\begin{description}
\item[\textbf{Algorithm APX}]
\item[\textbf{Input:}]  A complete graph $G=(V,E)$ with  a nonnegetive cost function $c$ on edges, two subsets $R \subset V$ and $R^{\prime} \subset R$, a partition $\mathcal{R}=\{R_1,R_2,\ldots,R_k\}$ of $R$, $R_i \cap R_j=\phi$, $i \neq j$,  
and $\mathcal{R}^{\prime}=\{R^{\prime}_1,R^{\prime}_2,\ldots,R^{\prime}_k\}$ of $R^{\prime}$, $R^{\prime}_i \subset R_i$, where the cost function is metric.
\item[\textbf{Output:}] A clustered selected-internal  Steiner tree  $\mathcal{T}$  for $R$ and  $R^{\prime}$ in $G$.
\item[1.] 
For each $R_i$,  use Algorithm \textit{$A_{mst}(G[R_i])$}~\cite{cormen,prime} to find a MST $T^s_i$ of  $G[R_i]$, $1\le i \le k$.  
\item[2.] For each  $T^s_i$, select any two vertices $u$ and $v$ in  $R_i \setminus R^{\prime}_i$, and then
use Algorithm \textit{$A_{h}(T^{s}_i,u,v)$} to find a Hamiltonian path $T^{h}_i$ between the two vertices $u$ and $v$  in the cube of $T^{s}_i$.
\item[3.]  Construct $G/R$ and let $\widetilde{R} = \{ r_i | 1 \le i \le  k \} $, where $r_i$ is the vertex resulted from the contraction of $R_i$.
\item[4.]  Use Algorithm  $\mathcal{A}_{stp}(G/R,\widetilde{R})$ to find  the inter-cluster tree that spans all vertices in $\widetilde{R}$.
\item[5.] Replace each $r_i$ with  $T^{h}_i$ to obtain a clustered selected-internal Steiner tree  $\mathcal{T}$. 
\end{description}

The result of this section is summarized in the following theorem.

\begin{theorem}
Algorithm APX is a $(\rho+4)$-approximation algorithm for the CSISTP. ~\label{Theorem1}
\end{theorem}
\begin{proof} 
We first analyze the time-complexity of Algorithm APX as follows.  For each $R_i$,  use Prim's Algorithm~\cite{cormen,prime} to find a  MST for $G[R_i]$ takes $O(|R_i|^2)$ time. Hence, step 1  runs $O(|R|^2)$ time. Also, step 2 can take $O(|R|^2)$ time~\cite{kara}. Step 3  and step 5 take $O(|V|^2)$ and $O(|V|)$ time, respectively. Hence, the time-complexity of Algorithm APX is  dominated by the cost of the step 3 for running the  $\rho$-approximation algorithm for the STP.  

Next, we prove the performance ratio of Algorithm APX. 
Let $T_{opt}$  be the optimal solution for the CSISTP  for $R$ and $R^{\prime}$ in $G$ and $T^{o}_i$ is its local tree of $R_i$, $1\le i \le k$.  We also let $\widehat{T}$ be a tree satisfying Lemma~\ref{ltree}--\ref{cstp}.
For each local tree $T^{h}_i$, we have $c(T^{h}_i) \le 2 c(T^{s}_i)$ by Lemma~\ref{hamratio}.
Since each $T^s_i$ is a MST of $G[R_i]$, by Lemma~\ref{ltree}, we have $c(T^{s}_i) \le c(\widehat{T}_i) \le 2 c(T^{o}_i)$. 
Then the cost of $\widehat{T}/R$ is greater than or equal to the cost of the optimal solution for the Steiner tree problem for $\widetilde{R}$ in the  graph $G/R$. 
 Step 5  runs a $\rho$-approximation algorithm to solve the Steiner tree problem for $\widetilde{R}$ in the  graph
 $G/R$. Hence,  $c(\mathcal{T}/R) \le \rho c(\widehat{T}/R) \le \rho  c(T_{opt})$ by Lemma~\ref{intercluster}--\ref{cstp}.
Finally, we have 
\begin{eqnarray*}
c(\mathcal{T})= \sum_{i=1}^k c(T^{h}_i)+c(\mathcal{T}/R) &\le& 2 \sum_{i=1}^k  c(T^{s}_i)+\rho  c(T_{opt})\\
                                                       &\le& 4 \sum_{i=1}^k  c(T^{o}_i)+\rho  c(T_{opt})
																									     \le (\rho+4) c(T_{opt}),																								
\end{eqnarray*}
 and the theorem is proved.
\end{proof}

\section{Conclusion}
In this paper, we have investigated the CSISTP.  
Then we have proposed an approximation algorithm  
with performance ratio of $(\rho+4)$ for the CSISTP on metric graphs. 
For future research, improving the performance ratio for the CSISTP  is an  immediate direction.

\appendix
\section{Appendix}
 Karaganis~\cite{kara} 
 proved that the cube of a connected graph $G$ with at least three vertices is 
\textit{Hamiltonian-connected}, i.e., there exists a Hamiltonian path between any two vertices. In this proof, Karaganis let $T$ be  a spanning tree of $G$ and 
construct a Hamiltonian path between any two vertices $v_a$ and $v_b$ in $T^3$ recursively, in which $T^3$ is the cube of $T$. 
We list the algorithm, denoted by $A_{h}(T,v_a,v_b)$, as follows.

\begin{description}
\item[\textbf{Algorithm $A_{h}(T,v_a,v_b)$}]
\item[\textbf{Input:}]  A tree $T=(V_T,E_T)$ with  two vertices $v_a$ and $v_b$ of $T$.
\item[\textbf{Output:}] A Hamiltonian path $T^h$ between $v_a$ and $v_b$ in $T^3$.
\item[]
 Repeat the following steps until the condition $|V_T|=1$ holds.
\begin{description}
\item [1.] {\bf If} $v_a$ and $v_b$ are not adjacent {\bf then}
\begin{description}
\item[1.1.] 
Find a path $P= (v_a, v_1, v_2, \ldots, v_b)$ between  
$v_a$ and $v_b$  in $T$.
\item[1.2.] 
Cut the edge $(v_a,v_1)$ in $T$ such that $T$ be separated  into two trees,   
say $T_{v_a}$ and $T_{v_b}$ respectively, containing $v_a$ and $v_b$.
\item[1.3.] {\bf If} the number of the vertex of $T_{v_a}$ is one {\bf then} let 
$v_{a^{\prime}}$ be $v_a$ {\bf else} let $v_{a^{\prime}}$ be a vertex adjacent to $v_a$ in $T_{v_a}$. 
\item[1.4.] Let $v_{b^{\prime}}$ be the vertex  $v_1$ in $T_{v_b}$.
 \end{description}
\item [\ \ \ ] {\bf Else} 
\begin{description}
\item[1.5.] 
Cut the edge $(v_a,v_b)$ in $T$ such that $T$ be separated  into two trees,    
say $T_{v_a}$ and $T_{v_b}$ respectively, containing $v_a$ and $v_b$.
 \item[1.6.] {\bf If} the number of the vertex of $T_{v_a}$ is one {\bf then} 
let $v_{a^{\prime}}$  be $v_a$ {\bf else} let $v_{a^{\prime}}$ be a vertex adjacent to $v_a$ in $T_{v_a}$.
 \item[1.7.] {\bf If} the number of the vertex of $T_{v_b}$ is one {\bf then} 
let $v_{b^{\prime}}$  be $v_b$ {\bf else} let $v_{b^{\prime}}$ be a vertex adjacent to $v_b$ in $T_{v_b}$.
 \end{description}
\item[2.] 
Call $A_{h}(T_{v_a},v_a,v_{a^{\prime}})$ to find a Hamiltonian path $P_{v_a}$ 
between $v_a$ and $v_{a^{\prime}}$ in $T_{v_a}^3$.
\item[3.] Call $A_{h}(T_{v_b},v_{b^{\prime}},v_b)$ to find a Hamiltonian path 
$P_{v_b}$ between $v_{b^{\prime}}$ and $v_b$ in $T_{v_b}^3$.
\item[4.]  
Connect $P_{v_a}$ and $P_{v_b}$ by the edge $(v_{a^{\prime}},v_{b^{\prime}})$.
\end{description}
\end{description}

 After running $A_{h}(T,v_a,v_b)$, Karaganis~\cite{kara} 
proved that 
there exists a Hamiltonian path between the two vertices $v_a$ and $v_b$ in $T^3$. 
By the triangle inequality and traversal shortcuts between the adjacent vertices in $T^h$, it is clear that $c(T^h) \le 2 c(T)$. 
Let \textit{F($|V_T|$)}  be the time complexity of $A_{h}(T,v_a,v_b)$ 
 on a tree $T$ with $|V_T|$ vertices. Then we can express $F(|V_T|)$  recursively as a recurrence relation :  
$F(|V_T|)=O(|V_T|)+F(|V_{T_{v_a}}|)+F(|V_{T_{v_b}}|)$ and  the solution of 
$F(|V_T|)$ is $O(|V_T|^2)$ time. 

\end{document}